\documentclass[a4paper,11pt]{amsart}
\usepackage{amsmath}
\usepackage{graphicx}
\usepackage{epstopdf}
\usepackage{amssymb,amscd,array}
\usepackage{color}
\usepackage{float}

\makeindex \makeatletter
\def\captionof#1#2{{\def\@captype{#1}#2}}
\makeatother

\newcounter{tablegroup}
\newcounter{subtable}[tablegroup]

\newtheorem{thm}{Theorem}[section]
\newtheorem{cor}[thm]{Corollary}
\newtheorem{lem}[thm]{Lemma}
\newtheorem{prop}[thm]{Proposition}
\newtheorem{defn}[thm]{Definition}
\newtheorem{rem}[thm]{\bf Remark}
\newtheorem{exe}[thm]{\bf Example}
\numberwithin{equation}{section}

\begin{document}
\title[On Limit Sets of Monotone Maps on Regular Curves]
{On Limit Sets of Monotone Maps on Regular Curves}

\author{Aymen Daghar and Habib Marzougui}

\address{Aymen Daghar, \ University of Carthage, Faculty
	of Science of Bizerte, (UR17ES21), ``Dynamical systems and their applications'', Jarzouna, 7021, Bizerte, Tunisia.}
\email{aymendaghar@gmail.com}
\address{ Habib Marzougui, University of Carthage, Faculty
	of Science of Bizerte, (UR17ES21), ``Dynamical systems and their applications'', 7021, Jarzouna, Bizerte, Tunisia}
\email{habib.marzougui@fsb.rnu.tn}

\begin{abstract}
We investigate the structure of $\omega$-limit (resp. $\alpha$-limit) sets for a monotone map $f$ on a regular curve $X$.
We show that for any $x\in X$ (resp. for any negative orbit $(x_{n})_{n\geq 0}$ of $x$), the $\omega$-limit set
 $\omega_{f}(x)$ (resp. $\alpha$-limit set $\alpha_{f}((x_{n})_{n\geq 0})$) is a minimal set. This also hold for $\alpha$-limit set $\alpha_{f}(x)$ whenever $x$ is not a periodic point.
  These results extend those of Naghmouchi \cite{n}
    established whenever $f$ is a homeomorphism on a regular curve and those of Abdelli \cite{a}
    , whenever $f$ is a monotone map on a local dendrite. Further results related to the basin of attraction of an infinite minimal set are also obtained.
\end{abstract}
\maketitle

\section{\bf Introduction}

In the last two decades, the study of dynamical systems in one dimensional continua has developed and become increasingly intensive due to the rich dynamical properties that one dimensional maps can exhibit.
Examples of continua studied include, graphs, dendrites and local dendrites (see for instance \cite{a}, \cite{HM}, \cite{aam}, \cite{Nag3}).
Recently regular curves has given a special attention (see for example, \cite{an}, \cite{ka}, \cite{ka2}, \cite{n}, \cite{n2}, \cite{Seidler}). These form a large class of continua which includes local dendrites. The Sierpi\'{n}ski triangle is a well known example of a regular curve which is not a local dendrite. Regular curves appear in continuum theory and also in other branches of Mathematics such as complex dynamics; for instance, the Sierpi\'{n}ski triangle can be realized as the Julia set of the complex polynomial $p(z)=z^2+2$ (see \cite{Bldevan}). Also dendrites appear naturally as the Julia set of a complex polynomial $p(z) = z^2 + c$, then for certain values of $c$ (e.g. $c=i$), the Julia set $J$ of $p$ is a dendrite (see \cite{dev}, p. 296). Recall that the Julia set of $p$ is $J(p) = \{z\in \mathbb{C}: (p^n(z))_{n\geq 1} \textrm{ is bounded} \}$.
In \cite{Seidler}, Seidler proved that every homeomorphism of a regular curve has zero topological entropy (later, this result was extended by Kato in \cite{ka} to monotone maps).\\

In the present paper, we deal with several questions/problems. \\

First, we are interested in whether every $\omega$-limit set is a minimal set. Recall that this was proved in two cases:

$-$ for homeomorphisms on regular curves (Naghmouchi \cite{n})

$-$ for monotone maps on local dendrites (Abdelli \cite{a}).

 In Theorem \ref{t21} we prove a more general result by showing that this is in fact true for monotone maps on regular curves. We do not know whether the above result can be extended to still a larger class of curves. However, let us mention that already for monotone maps (resp. homeomorphism) on 
 dendroids
, it can happen that an $\omega$-limit
 set is not a minimal set (Makhrova (\cite{mak}, Theorem 5), (resp. Naghmouchi \cite{Nag3}). The dendroids constructed therein are rational curves.
  Some consequences are derived from Theorem \ref{t21}, we prove e.g. in Section 4 that the basin of attraction of an infinite minimal set for monotone maps on regular curves is closed extending the one of \cite{n} for homeomorphisms on regular curves.\\

 Second, we are interested to the study of accumulated points of backward orbit, say $\alpha$-limit set. For the case of homeomorphisms on regular curves, any $\alpha$-limit set is a minimal set as a consequence from above. In Theorem \ref{t31}, we prove a more general result by showing that for monotone map $f$ on a regular curve, any $\alpha$-limit set $\alpha_{f}(x)$ of a non periodic point $x$ is a minimal set. However, we built an example of a monotone map $f$ on an infinite star for which $\alpha_{f}(x)$ is not a minimal set for some periodic point $x$.\\

 Third, we address the question of the equality between the set of recurrent points and the set of almost periodic points.
 This was proved in two cases:
 
 $-$ for homeomorphisms on regular curves (Naghmouchi \cite{n2})
 
 $-$ for monotone maps on local dendrites (Abdelli et al. \cite{aam})
\bigskip

\section{\bf Preliminaries}

Let $X$ be a compact metric space with metric $d$ and let $f: X\longrightarrow X$ be a continuous map. Let $\mathbb{Z},\ \mathbb{Z}_{+}$ and $\mathbb{N}$ be the sets of integers, non-negative integers and positive integers, respectively. For $n\in \mathbb{Z_{+}}$ denote by  $f^{n}$ the $n$-$\textrm{th}$ iterate of $f$; that is, $f^{0}=\textrm{identity}$ and $f^{n}=f\circ f^{n-1}$ if $n\in \mathbb{N}$. For any $x\in X$, the subset
$\textrm{Orb}_{f}(x) = \{f^{n}(x): n\in\mathbb{Z}_{+}\}$ is called the \textit{orbit} of $x$ (under $f$). A subset $A\subset X$ is called \textit{$f-$invariant} (resp. strongly $f-$invariant) if $f(A)\subset A$ (resp., $f(A)=A$); it is further called a \textit{minimal set} (under $f$) if it is closed, non-empty and does not contain any $f-$invariant, closed proper non-empty subset of $X$. We define the \textit{$\omega$-limit} set of a point $x$ to be the set:
\begin{align*}
\omega_{f}(x) & = \{y\in X: \liminf_{n\to +\infty} d(f^{n}(x), y) = 0\}\\
& = \underset{n\in \mathbb{N}}\cap\overline{\{f^{k}(x): k\geq n\}}.
\end{align*}

A point $x\in X$ is called:

$-$ \textit{Periodic} of period $n\in\mathbb{N}$ if
$f^{n}(x)=x$ and $f^{i}(x)\neq x$ for $1\leq i\leq n-1$; if $n = 1$, $x$ is called a \textit{fixed point} of $f$ i.e.
$f(x) = x$;


$-$ \textit{Almost periodic} if for any neighborhood $U$ of $x$ there is $N\in\mathbb{N}$ such that
$\{f^{i+k}(x): 0\leq i\leq N\}\cap U\neq \emptyset$, for all $k\in \mathbb{N}$.

$-$\textit{Recurrent} if $x\in \omega_{f}(x)$.


We denote by P$(f$), AP$(f$), R$(f$) and $\Lambda(f)$ the sets of periodic points, almost periodic points, recurrent points and the union of all $\omega$-limit sets of $f$, respectively. Define the space

$X_{\infty} = \displaystyle\bigcap_{n\in \mathbb{N}}f^{n}(X)$. From the definition, we have the following inclusions:
$$P(f)\subseteq \textrm{AP}(f)\subseteq R(f)\subseteq \Lambda(f)\subseteq X_{\infty}.$$

This was proved in two cases:

$-$ for homeomorphisms on regular curves (Naghmouchi \cite{n2})

$-$ for monotone maps on local dendrites (Abdelli et al. \cite{aam})

In the definitions below, we use the terminology from Nadler \cite{Nadler} and Kuratowski \cite{Kur}.

\begin{defn}[\cite{Kur}, p. 131]\label{def} \textit{Let $X, Y$ be two topological spaces.
		A continuous map $f: X\longrightarrow Y$ is said to be \textit{monotone} if for any connected subset $C$ of $Y$, $f^{-1}(C)$ is connected.}
\end{defn}

When $f$ is closed and onto, Definition \ref{def} is equivalent to that the preimage of any point by $f$ is connected (cf. \cite{Kur}, p. 131). Notice that $f^{n}$ is monotone for every $n\in \mathbb{N}$ when $f$ itself is monotone.

A \emph{continuum} is a compact connected metric space.
An \emph{arc} $I$ (resp. a \emph{circle}) is any space homeomorphic to the compact interval $[0, 1]$
(resp. to the unit circle $\mathbb{S}^{1} =\{z\in \mathbb{C}: \ \vert z\vert = 1\}$).
A space is called \textit{degenerate} if it is a single point, otherwise; it is \textit{non-degenerate}.
A \textit{dendrite} is a locally connected continuum which contains no circle.
A \textit{local dendrite} is a continuum every point of which has a dendrite neighborhood. By (\cite{Kur}, Theorem 4, p. 303), a local dendrite is a locally connected continuum containing only a finite number of
circles. As a consequence every sub-continuum of a local
dendrite is a local dendrite (\cite{Kur}, Theorems 1 and 4, p. 303). Every graph and every dendrite is a local
dendrite.

A \textit{regular curve} is a continuum $X$ with the property that for every point $x\in X$ and each open neighborhood $V$ of $x$ in $X$, there exists an open neighborhood $U$ of $x$ in $V$
such that the boundary set $\partial U$ of $U$ is finite.
Each regular curve is a $1$-dimensional locally connected continuum.
It follows that each regular curve is locally arcwise connected (for more details see \cite{Kur} and \cite{Nadler}). In particular every local dendrite is a regular curve (cf.
\cite{Kur}, p. 303).
A continuum $X$ is said to be \textit{rational} if any point in $X$ has a local base of open sets
with countable boundary. In particular every regular curve is a rational curve.\\

Let $X$ be a compact metric space. We denote by $2^X$ (resp. $C(X)$) the set of all non-empty compact subsets (resp. compact connected subsets) of $X$. The Hausdorff metric $d_H$ on $2^{X}$ (respectively $C(X)$) is defined as follows: $d_H (A,B) = \max \Big(\sup_{a\in A} d(a,B), \sup_{b\in B} d(b,A)\Big)$, where $A, B \in 2^X$ (resp. $C(X)$). For $x\in X$ and $M\in 2^X$, $d(x,M) = \inf_{y \in M} d(x,y)$. With this distance, $(C(X), d_H)$ and  $(2^X, d_H)$ are compact metric spaces. Moreover if $X$ is a continuum, then so are $2^{X}$ and $C(X)$ (for more details, see \cite{Nadler}). Let $f: X\longrightarrow X$ be a continuous map of $X$. We denote by $2^{f}:2^{X}\longrightarrow2^{X}, A\to f(A)$, called the induced map. Then $2^{f}$ is also a continuous self mapping of $(2^{X},d_{H})$ (cf. \cite{Nadler}). For a subset $A$ of $X$, we denote by diam$(A) = \sup_{x,y\in A} d(x,y)$.

A family $(A_{i})_{i\in I}$ of subsets of $X$ is called a \textit{null family} if for any infinite sequence $(i_{n})_{n\geq 0}$ of $I,\; \displaystyle\lim_{n\to +\infty}\textrm{diam}(A_{i_{n}})=0$.
\medskip
It is well known that 
each pairwise disjoint family of subcontinua of a regular curve is null (see \cite{LFS}). For the convenience of the reader, we give below a proof of it.

\begin{prop} \label{Fs}
Let $X$ be a regular curve. Then for any $\varepsilon>0$ and for any family of pairwise disjoint subcontinua $(A_{i})_{i\in I}$ of $X$, the set \\
$\{i\in I: \textrm{diam}(A_{i})\geq \varepsilon\}$ is finite. In particular if $(A_{n})_{n\geq 0}$ is a sequence of pairwise disjoint continua, then $(A_{n})_{n\geq 0}$ is a null family.
\end{prop}

\begin{proof}
Assume that for some $\varepsilon>0$, there exists an infinite family of pairwise disjoint continua $(A_{n})_{n\geq 0}$ such that for any $n\geq 0$, we have diam$(A_{n})>\varepsilon$. So for each $n\geq 0$, there exist $\{x_{n}, y_{n}\}\subset A_{n}$ such that $d(x_{n},y_{n})\geq \varepsilon$. As $X$ is a compact set, one can assume that the two sequences $(x_{n})_{n\geq 0}$ and $(y_{n})_{n\geq 0}$ converge to $x$ and $y$, respectively. Moreover $x\neq y$. So let $O_{x}$ be an open set with finite boundary such that $x\in O_{x}$ and $y\notin \overline{O_{x}}$. Then there exists some $N\geq 0$ such that for any $n\geq N$, we have $x_{n}\in O_{x}\cap A_{n}$. As $y\notin \overline{O_{x}}$, there exists $N_{1}>N$ such that for any $n\geq N_{1}$, we have $y_{n}\notin O_{x}$. Thus $A_{n}$ meets the finite boundary $\partial O_{x}$, this will lead to a contradiction since the $A_{n}, n\geq 0$ are disjoint. In particular, if $(A_{n})_{n\geq 0}$ is a sequence of pairwise disjoint continua, then
$\displaystyle\lim_{n\to +\infty}\textrm{diam}(A_{n})=0$.
\end{proof}

We will use the following lemma which is well known in the literature. 

\begin{lem}\label{lmcsf}$($\cite{Blokh}$)$
	Let $X$ be a compact metric space, $f: ~X \rightarrow X$ be a continuous map and let $\omega_f(x)$ be an $\omega$-limit
	set of a point $x\in X $. If $\omega_f(x)$ is the union of $N$ pairwise disjoint closed subsets
	$A_0, A_1, \dots, A_{N-1}$ which are cyclicly permuted by $f$ $($i.e. $f(A_i) = A_{i+1 (\textrm{mod}) N})$, then there exists $j\in \{0,1,\dots,N-1\}$ such that $\omega_{f^N}(x) = A_j.$
\end{lem}

In \cite{MI}, it was introduced the notion of totally periodic $\omega$-limit set. Recall that an $\omega$-limit set is said to be \textit{totally periodic } if it is composed of periodic points. A compact metric space $X$ has the \textit{$\omega$-FTP property} provided that for each self-continuous map $f$ of $X$, every totally periodic $\omega$-limit
set of $f$ is finite. In \cite{an}, it was proved that regular continua have the $\omega$-FTP property for homeomorphisms. Recently, the above result was relaxed to monotone maps in \cite{am}.

\begin{thm}$($\cite{am}$)$\label{c25}
	Let $X$ be a regular curve and $f: X\longrightarrow X$ be a monotone map.
	Then any totally periodic $\omega$-limit set of $f$ is finite.
\end{thm}
\bigskip

\section{\bf  The $\omega$-limit sets for monotone maps on regular curves}
\bigskip
The aim of this section is to prove the following theorem.

\begin{thm}\label{t21}
Let $X$ be a regular curve and $f$ a monotone self mapping of $X$. Then each $\omega$-limit set of $f$ is a minimal set.
\end{thm}

\begin{proof}
	Let $x\in X$. If $\omega_f(x)$ is finite, then it is a periodic orbit.
	Now assume that $\omega_f(x)$ is infinite and $\omega_f(x)$ is not minimal.
	Then there is a minimal set $M\subsetneq\omega_f(x)$. We distinguish two cases:
	\medskip
	
	\textbf{\textit{Case 1.}} \textit{$M$ is infinite}.
	Let $U$ be an open neighborhood of $M$ in $X$ with finite boundary of cardinality $k\in\mathbb{N}$
	and such that $\omega_f(x)\setminus \overline{U}\neq \emptyset$. Then there exist $k+1$ positive
	integers $n_0<\dots<n_{k}$ such that $f^{n_i}(x)\in X\setminus \overline{U}$, set $r=n_{k}-n_0$.
	Let $z\in M$, hence its orbit is infinite. So $f^{i}(z)\neq f^{j}(z)$, for $i\neq j$ and then as $X$ is locally arcwise connected, there exists a pairwise disjoint family $(W_{i})_{0\leq i\leq r}$ of arcwise connected open neighborhoods of $f^i(z)$: $0\leq i\leq r$. As $M\subset \omega_f(x)$, there exist some $m>n_0$ and a family of arcs $(J_{i})_{0\leq i\leq r}$ such that for any $0\leq i\leq r,\;J_{i}\subset W_i$ and $J_{i}$ join $f^{i}(z)$ and $f^{m+i}(x)$. As the arcs $J_{0},J_{1},\dots,J_{r}$ are pairwise disjoint, then so are the
	$f^{-(m-n_0)}(J), \dots, f^{-(m-n_0)}(J_{r})$. Set $C_i: = f^{-(m-n_0)}(J_{n_i-n_0})$, $0\leq i\leq k$. Then the $C_i$ are pairwise disjoint, $0\leq i\leq k$.
	Since $f$ is monotone, each  $C_i$ is connected. Moreover $C_i$ contains $f^{n_i}(x)$ and $f^{-(m-n_0)}(f^{n_i-n_0}(z))$.
	As $M = f^{p}(M)$, for any $p\geq 0$,
	$u_{i}:= f^{n_{i}-n_{0}}(z)\in M$ and there is $v_{i}\in M$ such that $f^{m-n_{0}}(v_{i})=u_{i}$. Thus $f^{-(m-n_0)}(f^{n_i-n_0}(z))$ contains $v_i$ and hence  $f^{-(m-n_0)}(f^{n_i-n_0}(z))\cap M \neq \emptyset$.
	This implies that $C_i$ intersects $U$ and
	$X\backslash U$, so it intersects the boundary $\partial U$ of $U$ for each $0\leq i\leq k$.
	Therefore $\partial U$ contains at least $k+1$ points, a contradiction.
	\medskip
	
	\textbf{\textit{Case 2.}} \textit{$M$ is finite}. In this case, $M = \textrm{Orb}_f(a)$ is a periodic orbit.
	
	\medskip
	
	Let $N$ be the period of $a$. By Lemma \ref{lmcsf}, $\omega_f(x) = \underset{0\leq i\leq N}\cup \omega_{f^{N}}(f^{i}(x))$, then
	$a\in \omega_{f^{N}}(f^{i}(x))$ for some $0\leq i\leq N$. Set $g= f^{N}$ and $y= f^{i}(x)$. Then $g$ is monotone, $g(a)=a$ and
	$a\in \omega_g(y)$. Moreover $\{a\}\subsetneq \omega_g(y)$; indeed, otherwise, $\{a\}=\omega_g(y)$ and so $\omega_f(y)=\omega_f(x)$
	is finite, a contradiction. Therefore without loss of generality, we may assume that
	$M=\{a\}$ is a single fixed point of $f$ with $\{a\}\subsetneq \omega_f(x)$.
	\medskip
	
	\textbf{Claim 1. } For each open neighborhood $V$ of $a$ in $X$ such that $\omega_f(x)\setminus \overline{V}\neq \emptyset$,
	there is a point $y_V\in \omega_f(x)\cap \overline{V}$ such that $y_{V}\neq \{a\}$ and $y_V\in f^{j}(\overline{V}\cap \omega_f(x))$, for all $j\in\mathbb{N}$.
	
	\begin{proof} Since $a\in \omega_f(x)$ and $\omega_f(x)\cap (X\setminus \overline{V}) \neq \emptyset$, there are an increasing sequence of positive integers $(n_i)_{i\geq 0}$
		and a sequence of integers $(t_i)_{i\geq 0}$ with the following properties:\\
		(i) for each $i\in \mathbb{Z}_+$, $t_i\geq i$ and $f^{n_i+j}(x)\in V$ for $0\leq j\leq t_i$\\
		(ii) $f^{n_i+t_i+1}(x)\notin V$.\\
		Since $X$ is a compact space, we may assume that the sequence $(f^{n_i+t_i}(x))_{i\geq 0}$ converges to a point $y_V$.
		Clearly, $y_V\in \omega_f(x)\cap \overline{V}$. Now let $j\in\mathbb{N}$.
		For every $i\geq j$, we have $0\leq n_i+t_i-j\leq n_i+t_i$ and so
		$f^{n_i+t_i}(x) = f^{j}(f^{n_i+t_i-j}(x))$. Let $y_{V,j}$ be some limit point of $(f^{n_i+t_i-j}(x))_{i\geq 0}$. Clearly $y_{V,j}\in \overline{V}\cap \omega_f(x)$ and $f^{j}(y_{V,j})=y_{V}$. Therefore for all  $j\in\mathbb{N}$, we have $y_V\in f^{j}(\overline{V}\cap \omega_f(x))$.\\
		On the other hand, if $y_{V}=a$, then $f(y_{V})=a$. As $(f^{n_i+t_i}(x))_{i\geq 0}$ converges to $y_{V}$, then $(f^{n_i+t_i+1}(x))_{i\geq 0}$ converges to $f(y_{V})=a$. Thus for $i$ large enough, we have $f^{n_i+t_i+1}(x)\in V$, which is a contradiction. This ends the proof of the claim.
	\end{proof}
	
	Since $X$ is a regular curve, there is an open neighborhood $U$ of $a$ in $X$ with finite boundary of cardinality $k\in\mathbb{N}$ and such that $\omega_f(x)\setminus \overline{U}\neq \emptyset$. Then there exist $k+1$ positive
	integers $n_0<\dots<n_{k}$ such that $f^{n_i}(x)\in X\setminus \overline{U}$. Set $r=n_{k}-n_0$. We distinguish two sub-cases.
	\medskip
	
	\textbf{Case 2.1.} For some open neighborhood $V$ of $a$ in $X$ such that
	$\overline{V}\subset U$, there exists a point $y:=y_V$ defined as in Claim 1 which is not periodic i.e. $y\notin P(f)$.
	Then the $f^{-i}(\{y\})$ ($0\leq i\leq r$) are pairwise disjoint non-empty connected compact subsets. Let $(W_{i})_{0\leq i\leq r}$ be a disjoint family of arcwise connected open neighborhoods of $f^{-i}(\{y\})$, $0\leq i\leq r$. Since $y\in \omega_f(x)$, there exists $b_{r}\in \omega_{f}(x)\cap f^{-r}(y)$. Then for any $0\leq i\leq r,\; f^{i}(b_{r})\in f^{-(r-i)}(\{y\})\cap\omega_f(x)\subset W_{r-i}\cap\omega_f(x)$. Hence there exists an integer $m>n_k$ such that $f^{m+i}(x)\in W_{r-i}$. Indeed, there exists $m_j\to +\infty$ so that $b_r = \underset{j\to +\infty}\lim f^{m_j}(x)$. Then  $f^{i}(b_r) = \underset{j\to +\infty}\lim f^{i+m_j}(x)$, for any $0\leq i\leq r$. It follows that there exists $j_i> n_k$ such that for all $j\geq j_i$, $f^{i+ m_j}(x)\in W_{r-i}$. It suffices to take $m= m_N$, where $N= \underset{0\leq i\leq r}\max j_i$, to get that $f^{i+ m}(x)\in W_{r-i}$. For any $0\leq i\leq r$, denote by $C_{i} = [f^{i}(b_{r}), f^{m+i}(x)]\cup f^{-(r-i)}(\{y\})$, where $[f^{i}(b_{r}), f^{m+i}(x)]$ is an arc in
	$W_{r-i}$ joining $f^{i}(b_{r})$ and $f^{m+i}(x)$. One has $C_i\subset W_{r-i}$. Since $f$ is monotone, $C_{i}$ is connected, for any $0\leq i\leq r$. Moreover
	the $C_{i}$ ($0\leq i\leq r$) are pairwise disjoint connected subsets, and among them $C_{n_{i}-n_{0}}$ ($ 0\leq i\leq k$).
	Set $K_i = f^{-(m-n_0)}(C_{n_{i}-n_{0}})$. Again by monotonicity of $f$, the $K_i$ ($0\leq i\leq k$) are pairwise disjoint
	connected sets. Observe that $K_i$ contains $f^{n_i}(x)$. On the other hand, $K_i$ contains $f^{-(m-n_0)}(f^{-(r-(n_{i}-n_{0}))}(y))=f^{-(m-n_{i}+r)}(y)$.
	As $y\in f^j(\overline{V})$, for all $j\in\mathbb{N}$, so there is $z_{i}\in \overline{V}$ such that
	$y= f^{m-n_{i}+r}(z_{i})$, since $m> n_k$. Thus $f^{-(m-n_0)}(f^{-(r-(n_{i}-n_{0}))})(y)$ contains $z_{i}$. Therefore $K_i\cap \overline{V}\neq \emptyset$.
	In result, $K_i$ intersects $X\setminus \overline{U}$
	and $U$, so it intersects the boundary $\partial U$, for $0\leq i\leq k$. This implies that $\partial U$ contains at least $k+1$ points. A contradiction.
	\medskip
	
	\textbf{Case 2.2.} For each open neighborhood $V$ of $a$ in $X$ such that
	$\overline{V}\subset U$, each point $y_V$ defined as in Claim 1 is periodic.
	\medskip
	
	For each $n\in\mathbb{N}$, let $V_{n}=B(a,\frac{1}{n})$ be the open ball centered at $a$ of radius $\dfrac{1}{n}$. One can assume that $\overline{V_{n}}\subset U$ and $\omega_f(x)\setminus \overline{V_n}\neq \emptyset$. By Claim 1, there exists an infinite sequence
	$(y_n)_{n\in \mathbb{N}}$ such that for any $n\in \mathbb{N}$, $y_{n} \in \omega_f(x) \cap \overline{V_n}$. Moreover, $y_n\neq a$ and $y_n\in f^{j}(\overline{V_n}\cap \omega_f(x))$, for all $n,j\in\mathbb{N}$. Clearly the sequence $(y_{n})_{n\in \mathbb{N}}$ converges to $a$ and moreover for any $n\geq 0$, $y_{n}\in P(f)$.
	\medskip
	
	For each $n\in \mathbb{N}$, denote by $p_{n}$ the period of $y_n$.
	\medskip
	
	\textbf{Claim 2}: The sequence $(p_{n})_{n\geq 0}$ is bounded.
	
	\begin{proof} On the contrary, there exists an increasing subsequence $(p_{n_{i}})_{i\geq 0}$ of $(p_{n})_{n\in \mathbb{N}}$. Since $f^{j}(y_n)$ converges to $a$, so for every $0\leq j\leq r$, there is $N>0$ such that $p_{N}>r$ and $\{y_{N},f(y_{N}),\dots, f^{r}(y_{N})\}\subset U$. Moreover $(f^{i}(y_N))_{0\leq i\leq r} \subset \omega_{f}(x)$ are pairwise distinct. Choose disjoint open arcwise connected neighborhoods $(W_i)_{0\leq i\leq r}$ of
		$f^{i}(y_N)$, $0\leq i\leq r$. Since $y_N\in \omega_f(x)$, there exist an integer $m>n_k$ such that $f^{m+i}(x)\in W_{i}$.
		Denote by $I_{i}=[f^{i}(y_N), f^{m+i}(x)]\subset W_{i}$ an arc
		 joining $f^{i}(y_N)$ and $f^{m+i}(x)$, $0\leq i\leq r$. We see that
		the $I_{i}$ ($0\leq i\leq r$) are pairwise disjoint connected subsets, and among them $I_{n_{i}-n_{0}}: 0\leq i\leq k$.
		Set $K_i = f^{-(m-n_0)}(I_{n_{i}-n_{0}})$. By monotonicity of $f$, the $K_i$ ($0\leq i\leq r$) are pairwise disjoint
		connected sets. Observe that $K_i$ contains $f^{n_i}(x)$. On the other hand, $K_i$ contains $f^{-(m-n_0)}(f^{n_i-n_0}(y_N))$. As $m-n_{0}>n_{i}-n_{0}$, we have
		\begin{align*}
		f^{-(m-n_0)}(f^{n_i-n_0}(y_N)) & = f^{-((m-n_{0})-(n_{i}-n_{0}))}(f^{-(n_{i}-n_{0})}(f^{n_i-n_0}(y_{N}))) \\
		& \supseteq  f^{-((m-n_{0})-(n_{i}-n_{0}))}(y_{N}) = f^{-(m-n_i)}(y_N).
		\end{align*}
		Recall that $f^{-j}(y_N)\cap \overline{V_N}\neq \emptyset$, for all $j\in\mathbb{N}$, hence $f^{-(m-n_i)}(y_N))\cap\overline{V_{N}} \neq \emptyset$, for all $0\leq i\leq k$. Therefore $K_i\cap \overline{V_N}\neq \emptyset$.
		In result, for $0\leq i\leq k$, $I_i$ intersects $X\setminus \overline{U}$
		and $U$, so it intersects the boundary $\partial U$. This implies that $\partial U$ contains at least $k+1$ points. A contradiction.
	\end{proof}
	
	Let $k\in \mathbb{N}$ such that $p_{n}\leq k$ for all $n$. Set $Y=\{y_{n}: n\in \mathbb{N}\}$. Then $Y$ is infinite with $Y\subset \textrm{Fix}(f^{k!})$. By considering $g=f^{k!}$ instead of $f$, we can assume that $Y\subset \textrm{Fix}(f)\cap \omega_{f}(x)$.
	Since $\omega_{f}(x)$ is infinite, so by Theorem \ref{c25}, $\omega_{f}(x)\nsubseteq P(f)$. Hence there is a point $z\in \omega_{f}(x)\setminus  P(f)$ and thus $(f^{-k}(z))_{k\geq 0}$ is pairwise disjoint. Fix a negative orbit $(z_{n})_n$ of $z$: $z_{0}=z$ and for any $n\geq 0$ $f(z_{n+1})=z_{n}$ with $z_{n}\in \omega_{f}(x)$ and denote by $Z$ the set of limit points of the sequence $(z_{n})_{n\geq 0}$. Clearly $Z\subset \omega_{f}(x)$.\\

	\textbf{Claim 3}: $Y\nsubseteq Z$.
	\begin{proof}
		Assume that $Y\subset Z$. Since $y_1\neq a$, there are disjoint open neighborhoods $O_{a}$ (resp. $O_{1}$) of $a$ (resp. $y_1$) of $X$ such that the boundary of $O_{a}$ is finite; set $s= \textrm{card}(\partial(O_a))$. As $(y_{n})_{n\geq 0}$ converges to $a$, there exists some $N>1$ such that $\{y_N,\dots y_{N+1}, \dots y_{N+s}\}\subset O_{a}$. As $X$ is locally arcwise connected and $Y\subset Z$, we can find $m_{0}<m_{1}<\dots<m_{s}$ and a family $(I_{j})_{0\leq j\leq s}$ of pairwise disjoint arcs in $O_a$ joining $y_{N+j}$ and $z_{m_j}$. As $y_{1} \in Z\cap \textrm{Fix}(f)$, then there is $p>0$ such that for any $0\leq j\leq s$, $z_{m_{j}+p}\in O_{1}$. By monotonicity of $f$, the $f^{-p}(I_{j})$ ($0\leq j\leq s$) are pairwise disjoint continua, each of them contains  $y_{N+j}\in O_a$ and  $z_{m_{j}+p}\in O_1$.  This implies that the boundary $\partial O_a$ contains at least $s+1$ points. A contradiction.
	\end{proof}
	\medskip
	
	\textit{End of the proof of Theorem \ref{t21}.} Let $y: =y_N\in Y\setminus Z$, for some $N>0$. We can find an open neighborhood $O_{y}$ of $y$ in $X$  with finite boundary such that $\overline{\{z_{n}:n\geq 0\}} \cap O_{y}=\emptyset$. Set $s=\textrm{card}(\partial O_{y})$. As $y\in \omega_{f}(x) \cap \textrm{Fix}(f)$ we can find $m>0$ such that $\{f^{m}(x),f^{m+1}(x),\dots, f^{m+s}(x)\} \subset O_{y}$. The negative orbit $(z_{n})_{n\geq 0}$ of $z$ is infinite, so $(z_{s},z_{s-1},\dots,z)$ is family of a pairwise distinct points of $\omega_{f}(x)$. We can find an integer $p_{1}>m$ and a family of pairwise disjoint arcs $(I_{j})_{0\leq j\leq s}$ each of which joins $z_{s-j}$ to $f^{p_{1}+j}(x)$. It follows that the $(f^{-(p_{1}-m)}(I_{j}))_{0\leq j\leq s}$ is family of pairwise disjoint continua each of which meets $O_{y}$ in at least $f^{m+j}(x)$ and meets $X\setminus O_{y}$ in at least $z_{s-j+p_{1}-m}$. Hence they meet the boundary $\partial O_a$ in at least $s+1$ points. A contradiction.
	In result, $\omega_f(x)$ is itself minimal.
	This ends the proof of Theorem \ref{t21}.
\end{proof}
\medskip

\begin{cor}[\cite{a}, \cite{n}] \label{c22}
If $f$ is a monotone map on a local dendrite (resp. $f$ is a homeomorphism on a regular curve), then each $\omega$-limit set of $f$ is a minimal set.
\end{cor}

 \begin{cor}\label{c23} Let $f$ be a monotone map on a regular curve $X$. Then $\Lambda(f)=R(f)=\textrm{AP}(f)$.
 \end{cor}
 \medskip

 \section{\bf The basin of attraction of an infinite minimal set}

Let $f$ be a monotone map on a regular curve $X$. If $M$ is an infinite minimal set of $f$, we call  $A(M)=\{x\in X: \ \omega_f(x)= M \}$ the \textit{basin of attraction of} $M$.
\medskip

\begin{thm}\label{amjde}
  Let $X$ be a regular curve and $f$ a monotone self mapping of $X$. If $M$ is an infinite minimal set of $f$, then $A(M)$ is $f$-invariant and closed in $X$.
\end{thm}

\begin{proof}
		From the equality $\omega_{f}(f(x))=\omega_{f}(x)$ for any $x\in X$, it follows that $A(M)$ is $f$-invariant.
		Let $(x_{n})_{n\geq0}$ be a sequence of $A(M)$ that converges to some $x\in X$. Assume that $x\notin A(M)$ that is $\omega_{f}(x)\neq M$. Then by Theorem \ref{t21}, $\omega_{f}(x)\cap M = \emptyset$. Set $L=\overline{O_{f}(x)}$. Then there exists an open neighborhood $U_{L}$ of $L$ with finite boundary such that $M\cap U_{L}=\emptyset$. We let $s = \textrm{card}(\partial U_L)$. As $f$ is continuous, we can find $N\geq 0$ such that $\{x_{N},f(x_{N}),\dots,f^{s}(x_{N})\} \subset U_{L}$.  Let $y\in M$, as $M$ is an infinite minimal set, then $\{y,f(y),\dots f^{s}(y)\}$ are pairwise distinct, so we may pick $s+1$ pairwise disjoint arcwise connected neighborhoods $(V_{f^{i}(y)})_{0\leq i\leq s}$ of $(f^{i}(y))_{0\leq i\leq s}$. Recall that $\omega_{f}(x_{N})=M$, then we can find $k>s$ and a family of pairwise disjoint arcs $(I_{j})_{0\leq j\leq s}$ in $V_{f^{i}(y)}$, each of which with endpoints $f^{k+j}(x_{N})$ and $f^{j}(y)$. By monotonicity of $f$, the $(f^{-k}(I_{j}))_{0\leq j\leq s}$ is a family of pairwise disjoint continua each of which meets $U_{L}$ in at least $f^{j}(x_{N})$ and meets $X\setminus U_{L}$ in at least a point of $M$ in $f^{-k}(f^{j}(y))$. Hence these continua meet the boundary of $U_{L}$ in at least $s+1$ points, which is a contradiction.\\
	This finish the proof of Theorem \ref{amjde}.
\end{proof}
\medskip

\begin{rem}\label{r1} \rm{ (1) In addition to that $A(M)$ is $f$-invariant, we also have that \\ $f^{-1}(A(M))\subset A(M)$.\\		
(2) Theorem \ref{amjde} is not true in general when $M$ is finite (see \cite{n})}.
\end{rem}
\medskip

\begin{cor} Let $X$ be a regular curve, $f$ a monotone self mapping of $X$ and let $M$ be an infinite minimal set of $f$. If $M$ is connected then so is $A(M)$.
\end{cor}

\begin{proof}
	Suppose that $A(M)$ is not connected then there exist at least two connected components $C_{1}$ and $C_{2}$ of $A(M)$. Let $x\in C_1$, $y\in C_2$.
	\medskip
	
	\textbf{Claim.} There exists an open set $U$ of $X$ with finite boundary such that $A(M)\subset U$ and for any sub-continuum $C$ of $X$ that contains $x, y$ we have $C\nsubseteq U$.
	
	Indeed, by Theorem \ref{amjde}, $A(M)$ is a compact subset of $X$. So we can find a decreasing sequence of open sets $(O_{n})_{n\geq 0}$ with finite boundaries such that $A(M)=\displaystyle\bigcap_{n\geq 0}O_{n}$. If for infinitely many $n\geq 0$, there exists a sub-continuum $K_{n}$ of $X$ containing $x,y$ such that $K_{n}\subset O_{n}$, then there is a sub-continuum $K\subset A(M)$ containing $x$ and $y$ (it suffices to select a suitable converging subsequence of $(K_{n})_{n\geq 0}$), a contradiction with the fact that either $K\subset C_1$ or $K\subset C_2$. Hence for some $N$ and for any sub-continuum $K$ of $X$ containing $x,y$, we have $K\nsubseteq O_{n}$ for any $n\geq N$. Choose then $U=O_N$. As $M$ is a sub-continuum of $X$, we can find a decreasing sequence of arcwise connected open sets $(U_{n})_{n\geq 0}$ such that $M = \displaystyle\bigcap_{n\geq 0}U_{n}$.
	Since $\omega_{f}(x) = \omega_{f}(y) = M \subset U_{n}$ for any $n\geq 0$, there exists $k_{n}\geq 0$ such that for any $k\geq k_{n}$ we have $\{f^{k}(x), f^{k}(y)\} \subset U_{n}$. Then we can find an increasing sequence of integers $m_{n}$ such that for any $n\geq 0$ we have $\{f^{m_{n}}(x),f^{m_{n}}(y)\} \subset U_{n}$. Let $I_n$ be an arc in $U_n$ joining $f^{m_{n}}(x)$ and $f^{m_{n}}(y)$. Then $f^{-m_{n}}(I_{n})$ is a sub-continuum of $X$ that contains $x$ and $y$. By the claim above,  $f^{-m_{n}}(I_{n})\nsubseteq U$. As $x, y\in U$, $f^{-m_{n}}(I_{n})$ meets the finite boundary $\partial U$. Therefore there exists $z\in \partial U$ such that $z\in f^{-m_{n}}(I_{n})$ for infinitely many $n$. Hence for infinitely many $n$, we have $f^{m_{n}}(z)\in I_{n}\subset U_{n}$ and so $\omega_{f}(z)=M$. It follows that $z\in A(M)\subset U$, this contradicts the fact that  $z\in \partial U$.
\end{proof}

 In \cite{n2}, Naghmouchi proved that homeomorphisms on regular curves without periodic points, have a unique minimal set. In \cite{Aym}, Daghar showed that this later is either a circle or homeomorphic to a Cantor set. In this setting, does this last result remains true for monotone maps on regular curves without periodic points?

A partial answer to the above question is given by the following.

\begin{cor} Let $X$ be a regular curve and $f$ a monotone self mapping of $X$,
	If $P(f)=\emptyset$, then there exist one or uncountably many minimal sets. Moreover at most a finite number of them are connected.
\end{cor}

\begin{proof}
Let $(M_{i})_{i\in I}$ be a family of minimal sets of $(X,f)$. As $P(f)=\emptyset$, then for any $i\in I,\; M_{i}$ is an infinite minimal set and by Theorem \ref{amjde}, $A(M_{i})$ is a closed set of $X$. As the $(A(M_{i}))_{i\in I}$ are disjoint and $X=\displaystyle\bigcup_{i\in I}A(M_{i})$, so by (\cite{Kur}, Theorem 6, p. 173), $I$ is either uncountable or reduced to a point. For the rest, assume that there is infinitely many connected minimal sets $(M_n)_{n>0}$. Then pick a sequence $(M_{n})_{n>0}$ of them converging (in the Hausdorff metric) to a continuum $C$ of $X$. Then by Proposition \ref{Fs}, we have $\displaystyle\lim_{n\to +\infty}\textrm{diam}(M_{n})=0$. Hence $C$ is reduced to a fixed point, which is a contradiction.
\end{proof}

 \section{\bf The $\alpha$-limit sets for monotone maps on regular curves}

The aim of this section is to investigate the $\alpha$-limit sets. First recall the definition of the $\alpha$-limit set. Let $X$ be a compact metric space, $f: X\longrightarrow X$ a continuous map and $x\in X$. The set $\alpha_{f}(x)= \displaystyle\bigcap_{k\geq 0}\displaystyle\overline{\bigcup_{n\geq k}f^{-n}(x)}$ is called the $\alpha$-limit set of $x$. Equivalently a point $y\in \alpha_{f}(x)$ if and only if there exist an increasing sequence of positive integers $(n_{k})_{k\in \mathbb{N}}$ and a sequence of points  $(x_{k})_{k\geq 0}$ such that
$f^{n_{k}}(x_{k})=x$ and $\displaystyle\lim_{k\rightarrow +\infty}x_{k}=y.$  It is much harder to deal with $\alpha$-limit sets since there are many choices for points in a backward orbit. When $f$ is a homeomorphism, $\alpha_{f}(x)= \omega_{f^{-1}}(x)$.
\medskip

 Balibrea et al. \cite{bgl} considered exactly one branch of the backward orbit as follows.

\begin{defn}
	Let $x\in X$. A \textit{negative orbit} of $x$ is a sequence $(x_n)_n$ of points in $X$ such that $x_0=x$ and $f(x_{n+1}) = x_n$, for every $n\geq 0$. The $\alpha$-limit set of $(x_n)_{n\geq 0}$ denoted by $\alpha_f((x_n)_n)$ is the set of all limit points of $(x_n)_{n\geq 0}$.
\end{defn}

It is clear that $\alpha_f((x_n)_n)\subset \alpha_{f}(x)$. The inclusion can be strict;

 Balibrea et al. \cite{bgl} gave an example of a pair ($X, \sigma$), where $\sigma$ is the shift map on $\{0,1\}^{\mathbb{N}}$, for any point $x$, $\alpha_{f}(x)= X$ and there are two negative orbits $(x_n)_n$ and $(y_n)_n$ of $x$ under $f$ such that $\alpha_f((x_n)_n)$ consists of a single point and $\alpha_f((y_n)_n) = X$.
In the present paper, we show that the above inclusions can also be strict for onto monotone maps on regular curves (see  Example \ref{CEM}). Notice that we have the following equivalence:\\

(i) $\alpha_{f}(x)\neq \emptyset$, (ii) $x\in X_{\infty}$.
In particular, if $f$ is onto then $\alpha_{f}(x)\neq \emptyset$ for every $x\in X$.

\subsection{The limsup of a particular sequence of compact sets } \ 

Let $f: X\longrightarrow X$ be a continuous map of a compact space $X$ and let $(G_{n})_{n\geq 0}$ be a sequence of compact subsets of $X$ satisfying $f(G_{n+1})\subset G_{n}$, for any $n\geq 0$. We denote by $\limsup G_n =\displaystyle\bigcap_{k\geq 0}\displaystyle\overline{\bigcup_{n\geq k}G_{n}}$. A point $y\in \limsup G_n$ if and only if there exists an increasing sequence of positive integers $(n_{k})_{k\in \mathbb{N}}$ and a sequence of points  $(x_{k})_{k\geq 0}$ such that
$x_{k}\in G_{n_{k}}$ and $\displaystyle\lim_{k\rightarrow +\infty}x_{k}=y$.

 In particular:
\medskip

(1) If $x\in X_{\infty}$ and $G_n= f^{-n}(x),n\geq 0$, then $f(G_{n+1})\subset G_n$, for every $n\geq 0$ and we have $\limsup G_n = \alpha_{f}(x)$. 

(2) If $(x_{n})_{n\geq 0}$ is a negative orbit of $x$ and $G_n= \{x_n\}$, $n\geq 0$, then $f(G_{n+1})=G_{n}$ and we have $\limsup G_n= \alpha((x_{n})_{n})$.\\


First, we establish the following lemma which will be useful for most of the proofs in this section.

\begin{lem}\label{FI}
Let $(G_{n})_{n\geq 0}$ be a null sequence of compact subsets of $X$ satisfying $f(G_{n+1})\subset G_n$, for every $n\geq 0$. Then:
\begin{itemize}
	\item[(i)] If $A\in 2^{X}$ and for some sequence of positive integers $(n_{i})_{i\geq 0}$ we have $\displaystyle\lim_{i\to +\infty}d_{H}(A, G_{n_{i}})=0$, then $A$ is degenerate.
\item[(ii)] $x\in \limsup G_n$ if and only if there exists a sequence of positive integers $(n_{i})_{i\geq 0}$ such that $\displaystyle\lim_{i\to +\infty}d_{H}(x, G_{n_{i}})=0$.
\item[(iii)] If $(G_{n_{i}})_{i\geq 0}$ converges to $\{x\}$, then for any sequence of non-empty compact subsets $(A_{n_{i}})_{i\geq 0}$ of $X$ such that $A_{n_{i}}\subset G_{n_{i}}$ for any $i\geq 0$, the sequence $(A_{n_{i}})_{i\geq 0}$ also converges to $\{x\}$.
\end{itemize}
\end{lem}

\begin{proof} \rm{(i)} Let $a, b\in A$. As $\displaystyle\lim_{i\to +\infty}d_{H}(A, G_{n_{i}})=0$, so there exist $a_{i}, b_{i}\in G_{n_{i}}$ such that $\displaystyle\lim_{i\to +\infty}d(a_{i}, a)=0$ and $\displaystyle\lim_{i\to +\infty}d(b_{i}, b)=0$. Since $(G_{n})_{n\geq 0}$ is a null family, $\displaystyle\lim_{i\to +\infty}\textrm{diam}(G_{n_{i}})=0$. For any $i\geq 0$, we have $d(a_{i},b_{i})\leq \textrm{diam}(G_{n_{i}})$, thus $d(a,b)=0$ and so $a=b$. In result $A$ is degenerate.\\
\rm{(ii)} Let $x\in \limsup G_n$, then there exists a sequence of positive integers $(n_{i})_{i\geq 0}$ such that $x_{i}$ converges to $x$ with $x_{i}\in G_{n_{i}}$. We have $$ d_{H}(x, G_{n_{i}})\leq d_{H}(x, x_{i}) + d_{H}(x_{i},G_{n_{i}})\leq d_{H}(x, x_{i})+\textrm{diam}(G_{n_{i}}).$$ Since $\displaystyle\lim_{i\to +\infty}\textrm{diam}(G_{n_{i}}) = 0$, $\displaystyle\lim_{i\to +\infty}d_{H}(x, G_{n_{i}})=0$. The converse is trivial.\\
\rm{(iii)} We have $$d_{H}(x,A_{n_{i}}) \leq d_{H}(x,G_{n_{i}})+d_{H}(G_{n_{i}},A_{n_{i}}) \leq d_{H}(x, G_{n_{i}}) +\textrm{diam}(G_{n_{i}}).$$ As $\displaystyle\lim_{i\to +\infty}\textrm{diam}(G_{n_{i}}) = 0$ and $\displaystyle\lim_{i\to +\infty}d_{H}(x, G_{n_{i}})=0$, thus $\displaystyle\lim_{i\to +\infty}d_{H}(x,A_{n_{i}})=0$.
\end{proof}

\begin{lem} \label{l42} Let $f$ be a continuous self mapping of a compact space $X$ and $(G_{n})_{n\geq 0}$ be a null sequence of compact subsets satisfying $f(G_{n+1})\subset G_n$, for every $n\geq 0$. Then $\limsup G_n$ is non-empty, closed and strongly $f$-invariant.
\end{lem}	
	
\begin{proof} By compactness of $X$, it is obvious that $\limsup G_n$ is non-empty.
 Let $x\in \limsup G_n$. Then by Lemma \ref{FI}, there exists a sequence of positive integers $(n_{i})_{i\geq 0}$ such that $G_{n_{i}}$ converges to $\{x\}$. Consider now the sequence $(G_{n_{i}+1})_{i\geq 0}$. Since $(G_n)_{n\geq 0}$ is a null family, we can assume that $(G_{n_{i}+1})_{i\geq 0}$ converges to some point $\{z\}$. By Lemma \ref{FI}, $z\in \limsup G_n$. By continuity of $2^{f}$, we have $(f(G_{n_{i}+1}))_{i\geq 0}$ converges to $\{f(z)\}$. As $f(G_{n_{i}+1})\subset G_{n_{i}}$ then by Lemma  \ref{FI}, $f(z)=x$ and thus $\limsup G_n\subset f(\limsup G_n)$. The other inclusion can be proved by the same argument since $f(G_{n_{i}})\subset G_{n_{i}-1}$.
\end{proof}

\begin{cor}\label{c43}
	 Let $f: X\longrightarrow X$ be a continuous self mapping of a compact space $X$. Let $x\in X_{\infty}$. Then:
\begin{itemize} 	
\item[(i)] $\alpha_{f}(x)$ is non-empty, closed and $f$-invariant. In addition, it is strongly $f$-invariant whenever $\displaystyle\lim_{n\to +\infty}\textrm{diam}(f^{-n}(x)) =0$.
\item[(ii)] $\alpha_f((x_n)_{n\geq 0})$ is non-empty, closed and strongly $f$-invariant, for any negative orbit $(x_{n})_{n\geq 0}$ of $x$.
\end{itemize}
\end{cor}
\medskip


\subsection{Case of monotone maps on regular curves} The aim of this subsection is to prove the following theorem.

\begin{thm}\label{tG31}
	Let $X$ be a regular curve and $f$ a monotone self mapping of $X$. Let $(G_{n})_{n\geq 0}$ be a null family of compact sets of $X$ such that $f(G_{n+1})\subset G_{n}$, for every $n\geq 0$. Then $\limsup G_n$ is a minimal set.
\end{thm}
\medskip

 First we establish the following proposition.
\medskip

\begin{prop}\label{oinf} Let $X$ be a regular curve and $f$ a monotone self mapping of $X$. If $x\in X_{\infty}$ and $\omega_f(x)$ is infinite then $\omega_f(x) = \alpha_f(x)$.
\end{prop}

\begin{proof} Set $M = \omega_f(x)$. By Theorem \ref{amjde} and Remark \ref{r1}, $A(M)$ is a closed and invariant under $f$ and $f^{-1}$. As $x\in A(M)$, then $\alpha_{f}(x)\subset A(M)$. As $\alpha_{f}(x)$ is non-empty and $f$-invariant (Corollary \ref{c43}) then
	$M\subset \alpha_{f}(x)$. Now assume that  $M\subsetneq\alpha_f(x)$, then there exists some $a\in \alpha_{f}(x)\setminus M$. Let $O_{M}$ be an open neighborhood of $M$ with finite boundary of cardinality $s$ such that $a\notin \overline{O_{M}}$. Then there exist $n_{s}<n_{1}<\dots<n_{0}$ such that $\{x_{n_{j}}:\;0\leq j\leq s\}\cap O_{M}=\emptyset$ with $f^{n_{j}}(x_{n_{j}})=x$. Set $r=n_{0}-n_{s}$.
Let $t\in M$. As $M$ is infinite, we can find an integer $k\geq n_{0}$ such that $f^{k+j}(x)\in O_{M}$ for all $0\leq j\leq r$ and a family of pairwise disjoint arcs $(I_{j})_{0\leq j\leq r}\subset O_M$, each of which joins $f^{k+j}(x)$ and $f^{j}(t)$. Consider now the pairwise disjoint family $(f^{-(n_{0}+k)}(I_{j}))_{0\leq j\leq r}$, each member of it meets $M$ and contains $x_{n_{0}-j}$ (since $x_{n_{0}-j}\in f^{-(n_{0}+k)}(f^{k+j}(x))$). On the other hand, for $j=n_{0}-n_{i}, ;\ 0\leq i\leq s$, $f^{-(n_{0}+k)}(I_{j})$ contains $x_{n_{i}}$, hence it intersects $X\setminus O_{M}$. Therefore for $j=n_{0}-n_{i}, ;\ 0\leq i\leq s,\; f^{-(n_{0}+k)}(I_{j})$ meets the boundary $\partial O_{M}$. This contradicts card($\partial O_{M}) = s$.
\end{proof}
\medskip

\begin{proof}[Proof of Theorem \ref{tG31}]
	 We know by Lemma \ref{l42} that $\limsup G_n$ contains a minimal set. Assume that $\limsup G_n$ is not a minimal set, thus for any minimal set $M$, we have $M\neq \limsup G_{n}$.
	 \medskip
	
\textbf{Claim 1.} \textit{Any minimal set $M\subsetneq \limsup G_n$ is finite}.

\
\\
\textit{Proof}. On the contrary, there exists some infinite minimal set $M\subsetneq \limsup G_n$.
\smallskip

$\bullet$ First, let us prove that for any $t\in \limsup G_n$, we have $\omega_{f}(t)=M$:
	
 Assume that there is $t\in \limsup G_n$ such that $\omega_{f}(t)\neq M$. So $\omega_{f}(t)$ is a minimal set by Theorem \ref{t21}. Let $U$ be an open neighborhood of $M$ in $X$ with finite boundary of cardinality $k$ such that $\omega_{f}(t)\cap \overline{U}=\emptyset$ and let $z\in M$. Since the orbit $O_{f}(z)$ is infinite and $X$ is locally connected, there exists a connected closed neighborhood $K_{j}$ of $f^{j}(z)$ such that
 $(K_{j})$ $(0\leq j\leq k)$ are pairwise disjoint. Thus we can find $m>k$ such that $G_{m-j}\subset K_{j}$ for any $0\leq j\leq k$.
Now as $\omega_{f}(t)$ is a minimal set in $\limsup G_n\setminus U$, we can find $n>m$ such that $G_{n-j} \cap U=\emptyset$, for any $0\leq j\leq k$.
  Therefore $f^{-(n-m)}(K_{j})$ ($0\leq j\leq k$) is a family of pairwise disjoint continua each of which meets $U$ at some point in $f^{-(n-m)}(f^{j}(z))\cap M$ and meets $X\setminus U$ at $G_{n-j}$ since $G_{n-j} \subset f^{-(n-m)}(G_{m-j})$. Hence $f^{-(n-m)}(K_{j})$ meets the boundary $\partial U$ for any $0\leq j\leq k$. A contradiction. We conclude that for any $t\in \limsup G_n$, we have $\omega_{f}(t)=M$.
	\medskip
	
 $\bullet$ Second, let $t\in \limsup G_n\backslash M$. By above we have $\omega_{f}(t)=M$. As $\limsup G_n$ is strongly invariant (Lemma \ref{l42}), so $t\in X_{\infty}$. Now by Proposition \ref{oinf}, we have $\alpha_{f}(t)=M$. Let $V_{t}$ and $V_{M}$ be two disjoint open neighborhoods of $t$ and $M$ respectively with finite boundary. Since $M\subset V_{M}$ there exists some $N>0$ such that $f^{-n}(t) \subset V_{M}$ for any $n\geq N$.
 As $X$ is locally connected, let $V_{n}$ be an open connected neighborhood of $t$ such that $V_{n}\subset B(t,\frac{1}{n})$. Clearly $(\overline{V_{n}})_{n\geq 0}$ is a sequence of continua which converges to $\{t\}$ (with respect to the Hausdorff metric).
   As $t\in \limsup G_n$, then for any $n\geq 0$, we can find two increasing sequences of positive integers $(m_{n})_{n\geq 0}$ and $(s_{n})_{n\geq 0}$ such that $G_{m_{n}}\subset V_{n}$ and $G_{m_{n}+s_{n}}\subset V_{t}$. We can assume that $s_{n}>N$ for any $n\geq 0$. Then $f^{-s_{n}}(V_{n})$ meets $V_{M}$ at some point of $f^{-s_{n}}(t)$ and $V_{t}$ at $G_{m_{n}+s_{n}}$ since $G_{m_{n}+s_{n}}\subset f^{-s_{n}}(G_{m_{n}})$. Hence $f^{-s_{n}}(V_{n})$ meets the finite boundary $\partial V_{t}$, and then it meets one point $b\in \partial V_{t}$ for infinitely many $n$. Thus $f^{s_{n}}(b)\in V_{n}$ for infinitely many $n$ and then $t\in \omega_{f}(b)$. Since $\omega_{f}(b)$ is minimal, it follows that $\omega_{f}(t)=\omega_{f}(b)=M$. So $t\in M$, a contradiction. This ends the proof of Claim $1$. 
   \smallskip

\textbf{Claim 2}. Let $M$ be a finite minimal set in $\limsup G_n$. Then for each open neighborhood $V_{M}$ of $M$ in $X$ such that $\limsup G_n\setminus \overline{V_{M}}\neq \emptyset$, there is a point $y_V\in \limsup G_n\cap (\overline{V_{M}}\setminus M)$ such that $y_V\in f^{j}(\overline{V_{M}}\cap \limsup G_n)$, for all $j\in\mathbb{N}$.\\

\textit{Proof}.
Let $V_{M}$ be open neighborhood of $M$ in $X$ such that $\limsup G_n\setminus \overline{V_{M}}\neq \emptyset$. Let $m\in M$ and $t\in \limsup G_n\setminus V_{M}$. There exist two increasing sequences of integers $(n_{i})_{i\geq 0}$ and $(p_{i})_{i\geq 0}$ such that:

\rm{(1)} $(G_{n_{i}})_{i\geq 0}$ converges $\{m\}$.

\rm{(2)} $(G_{p_{i}})_{i\geq 0}$ converges to $\{t\}$.

\rm{(3)}  $n_{i}>p_{i}$ for any $i\geq 0$.

Denote by $A_{i,j}= \displaystyle\bigcup_{s=0}^{j}G_{n_{i}-s}$. Since $f(M)=M$ and $(G_{n})_{n\geq 0}$ is a null family, we can easily prove that for any $j\geq 0$, $(A_{i,j})_{i\geq 0}$ converges to $L\subset M$ (with respect to the Hausdorff metric). Thus for any $j\geq 0$, we can find $i_{j}\geq j$ such that $A_{i_{j},j}\subset V_{M}$. Since $(G_{n})_{n\geq 0}$ is a null family, so by (2), we can assume that $G_{p_{i_{j}}}\cap V_{M}=\emptyset $ and by (3), $n_{i_{j}}>p_{i_{j}}$. Therefore there exists $t_{i_{j}}\geq i_{j}$ such that:\\
 \rm{(i)} For any $0\leq s\leq t_{i_{j}},\; G_{n_{i}-s}\cap V_{M} \neq \emptyset$;\\
 \rm{(ii)} $G_{n_{i_{j}}-t_{i_{j}}-1} \cap V_{M}=\emptyset$;\\
 \rm{(iii)} $n_{i_{j}}-t_{i_{j}}-1\geq p_{i_{j}}$.\\

By (iii), we have $\displaystyle\lim_{j\to +\infty}n_{i_{j}}-t_{i_{j}}=+\infty$. So we can assume that $G_{n_{i_{j}}-t_{i_{j}}}$ converges (with respect to the Hausdorff metric). As $(G_{n})_{n\geq 0}$ is a null family, so by Lemma \ref{FI}, $\displaystyle\lim_{j\to +\infty}G_{n_{i_{j}}-t_{i_{j}}} = \displaystyle\lim_{j\to +\infty}\overline{ G_{n_{i_{j}}-t_{i_{j}}}\cap V_{M}} = \{y_{V}\}$. Hence $y_{V}\in \limsup G_n\cap \overline{V_{M}}$. Let $k\geq 0$ and let $\{y_{V,k}\}$ be some limit point of $(G_{n_{i_{j}}-t_{i_{j}}+k})_{j\geq 0}$. Let us prove that $ f^{k}(y_{V,k})=y_{V}$ and $y_{V,k}\in \limsup G_n\cap \overline{V_{M}}$.\\
 For $j$ large enough and from (i), we have $G_{n_{i_{j}}-t_{i_{j}}+k}\cap V_{M}\neq \emptyset$. Again by Lemma \ref{FI} and the fact that $(G_{n})_{n\geq 0}$ is a null family, we have $y_{V,k} \in \limsup G_n\cap \overline{V_{M}}$. Furthermore, $f^{k}(G_{n_{i_{j}}-t_{i_{j}}+k})\subset G_{n_{i_{j}}-t_{i_{j}}}$. Since $2^{f}$ is continuous, $\{f^{k}(y_{V,k})\}=\displaystyle\lim_{j\to +\infty}f^{k}(G_{n_{i_{j}}-t_{i_{j}}+k})\subset G_{n_{i_{j}}-t_{i_{j}}}$. As $G_{n_{i_{j}}-t_{i_{j}}}$ converges to $\{y_{V}\}$, we have $f^{k}(y_{V,k})=y_{V}$. In addition, $y_V\in f^{k}\Big(\overline{V_{M}}\cap \limsup G_n\Big)$, for all $k\in\mathbb{N}$.\\
 It remains to show that $y_{V}\notin M$. Otherwise, $y_{V}\in M$ and then $f(y_{V})\in M$. Thus for $j$ large enough, $ G_{n_{i_{j}}-t_{i_{j}}-1}\cap V_{M}\neq \emptyset$, which contradicts (ii). This ends the proof of Claim 2.
\medskip

 \textbf{Claim 3}. There is an open neighborhood $V$ of $M$ such that \\ $\limsup G_{n}\setminus \overline{V}\neq \emptyset$ and $y_{V}\notin P(f)$.
 \medskip

\textit{Proof.} Otherwise, for any open neighborhood $V$ of $M$ such that $\limsup G_{n}\setminus \overline{V}\neq \emptyset$, we have $y_{V}\in P(f)$. Let $(V_{n})_{n\geq 0}$ be a decreasing sequence of neighborhoods of $M$ such that $\limsup G_{n}\setminus \overline{V_n}\neq \emptyset$. So let $y_{n}=y_{V_{n}}$ given by Claim 2. Since for any $n\geq 0,\; y_{n}\notin M$, we can assume that the sequence $(y_{n})_{n\geq 0}$ is pairwise distinct. Let $W_{M}$ be some open neighborhood of $M$ with finite boundary of cardinality $k$ such that $O_{f}(y_{0})\cap \overline{W_{M}}=\emptyset$. In $W_{M}$ we can find infinitely many $(\overline{V_{n}})_{n\geq 0}$, so let $(O_{j})_{0\leq j \leq k}$ be a family of pairwise disjoint open connected neighborhoods of $y_{p_{j}}$. Recall that for any $0\leq j \leq k$, we have $y_{p_{j}}\in \limsup G_{n}$. Hence we can find $k+1$ positive integer $m_{k}<m_{k-1}<\dots<m_{0}$, such that $G_{m_{j}}\subset O_{j}$, for any $0\leq j \leq k$. Let $r = m_{0}-m_{k}$. Now as $O_{f}(y_{0})\subset\limsup G_n$, we can find $p>m_{0}$ such that $\displaystyle\bigcup_{s=0}^{r}G_{p-s}\cap W_{M}=\emptyset$. Therefore, for $i:=m_{0}-m_{j},\; 0\leq j\leq k,\;   f^{-(p-m_{0})}(O_{i})$ meets $\overline{V_{y_{n_{j}}}}\subset W_{M}$ at some point of $f^{-(p-m_{0})}(y_{p_{j}})$ and $X\setminus W_{M}$ at least in $G_{p-(m_{0}-m_{j})}$, thus it meets $\partial W_{M}$. A contradiction with $k=\textrm{card}(\partial W_{M})$.
\medskip
\smallskip

 \textit{End of the proof of Theorem \ref{tG31}.}
 Let $V$ and $y_{V}\notin P(f)$ be as in Claim 3. Since $\omega_{f}(y_{V})\subset \limsup G_{n}$ is a minimal set, so by Claim 1, it is finite. Since $y_{V}\notin P(f)$, $y_{V}\notin \omega_{f}(y_{V})$. Also $(f^{-k}(y_{V}))_{k\geq 0}$ is pairwise disjoint, so by Proposition \ref{Fs}, it is a null family. By Claim 2,  $f^{-k}(y_{V})\cap \overline{V}\neq \emptyset$, for any $k\geq 0$. Therefore $\alpha_{f}(y_{V})\subset \overline{V}$. Recall that $\limsup G_{n} \nsubseteq \overline{V}$, so let $t\in \limsup G_n\setminus \alpha_{f}(y_{V})$.\\
 We can find a sequence of positive integers $(n_{i})_{i\geq 0}$ such that $G_{n_{i}}$ converges to $y_{V}$. As $X$ is locally arcwise connected, we can find a decreasing sequence $(O_{i})_{i\geq 0}$ of connected open neighborhoods of $y_{V}$. Clearly $(\overline{O_{i}})_{i\geq 0}$ is a sequence of continua that converges to $\{y_{V}\}$ and we can assume that $G_{n_{i}}\subset O_{i}$ for any $i\geq 0$.
	Now as $t\notin \alpha_{f}(y_{V})$, we can find two disjoint open neighborhoods with finite boundary $V_{t}$ and $V_{\alpha_{f}(y_{V})}$. Clearly there exists some $N\geq 0$ such that $f^{-k}(y_{V})\subset V_{\alpha_{f}(y_{V})}$ for any $k\geq N$.
  Recall that $t\in \limsup G_{n}$. Then we can find infinitely many $n\geq 0$ such that $G_{n}\subset V_{t}$ and so we can find an increasing sequence $(s_{i})_{i\geq 0}$ such that $G_{n_{i}+s_{i}}\subset V_{t}$ for any $i\geq 0$. Clearly we can assume that for any $i\geq 0$ we have $s_{i}>N$. It follow that for any $i\geq 0$, $f^{-s_{i}}(O_{i})$ meets $V_{t}$ and $V_{\alpha_{f}(y_{V})}$, hence it meets the finite boundary $\partial V_{t}$. Thus there exists $b\in \partial V_{t}$ such that $f^{s_{i}}(b)\in O_{i}$ for infinitely many $i$. Therefore $y_{V}\in \omega_{f}(b)$ and hence $\omega_{f}(b) = \omega_{f}(y_{V})$  by minimality of $\omega_{f}(b)$. In result, $y_{V}\in \omega_{f}(y_{V})$. A contradiction.
\end{proof}

As a consequence, we obtain the following theorem.

\begin{thm}\label{t31}
	Let $X$ be a regular curve and $f$ a monotone self mapping of $X$. Then we have:
	\begin{itemize}
		\item[(i)] For any $x\in X_{\infty}\setminus P(f),\; \alpha_{f}(x)$ is a minimal set.
		\item[(ii)] For any $x\in X_{\infty}$, if $(x_{n})_{n\geq 0}$ is a negative orbit of $x$, then $\alpha((x_{n})_{n})$ is a minimal set.
	\end{itemize}
\end{thm}

	\begin{proof}
		\rm{(i)} If $x\in X_{\infty}\setminus P(f)$ and $G_n= f^{-n}(x)$, then $\limsup G_n= \alpha_{f}(x)$. Moreover $(G_n)_{n\geq 0}$ is a family of pairwise disjoint sub-continua of $X$ and so by Proposition \ref{Fs}, it is a null family. Furthermore we have $f(G_{n+1})\subset G_n$, for every $n\geq 0$. We deduce from Theorem \ref{tG31} that $\alpha_{f}(x)$ is a minimal set.
	
	\rm{(ii)} Set $G_n= \{x_n\}$, $n\geq 0$. In this case, $\limsup G_n= \alpha((x_{n})_{n\geq 0})$. Clearly  $(G_{n})_{n\geq 0}$ is a null family of compact sets and $f(G_{n+1})=G_{n}$. Thus from Theorem \ref{tG31}, $\alpha((x_{n})_{n\geq 0})$ is a minimal set.
	\end{proof}

\begin{cor}\label{c45} Let $X$ be a regular curve and $f$ a monotone self mapping of $X$ and $x\in X$.
	If $x\notin P(f)$ and ($x_{n})_{n\geq 0}$ is a negative orbit of $x$, then $\alpha_{f}(x) = \alpha_f((x_n)_{n\geq 0})$.
\end{cor}

\begin{proof}
  As $\alpha_f((x_n)_{n\geq 0})\subset\alpha_{f}(x)$, and $\alpha_{f}(x)$ is a minimal set (Theorem \ref{t31}), so Corollary \ref{c45} follows.
\end{proof}

\begin{rem} \rm{
		
		(i) The Corollary \ref{c45} is false in general whenever $x\in P(f)$ (see Example \ref{CEM}).
		
		(ii) Contrarily to homeomorphisms (see \cite{n}, Corollaries 2.5 and 2.6), it may happen, for a monotone map $f$ on a regular curve, that $\omega_{f}(x)$ is finite but $\alpha_{f}(x)$ infinite (see Example \ref{CEM}).}
\end{rem}

\begin{exe}\label{CEM} \rm{ We shall construct a monotone map $f$ on an infinite star $X$ centered at a point $z_{0}\in \mathbb{R}^{2}$ and beam $I_{n},\; n\geq 0$ with endpoints  $z_{0}$ and $z_{n}$ satisfying the following properties.
		\medskip
		
		\begin{itemize}
			\item[(i)] For any $x\in X\setminus \{z_{n}: n\geq 0\}$ we have $\omega_{f}(x)=\{z_{0}\}$ and $\alpha_{f}(x) =\{z_{n_{x}}\}$, where $n_{x}\geq 1$ such that $x\in I_{n_{x}}$.
			\item[(ii)] For any $n\geq 1$ we have $\alpha_{f}(z_{n})=\omega_{f}(z_{n})=\{z_{n}\}$.
			\item[(iii)] $\omega_{f}(z_{0})=\{z_{0}\},\; \alpha_{f}(z_{0})=X$.
		\end{itemize} }
	\end{exe}
	
	\begin{figure}[!h]
		\centering
		\includegraphics[width=12cm, height=10cm]{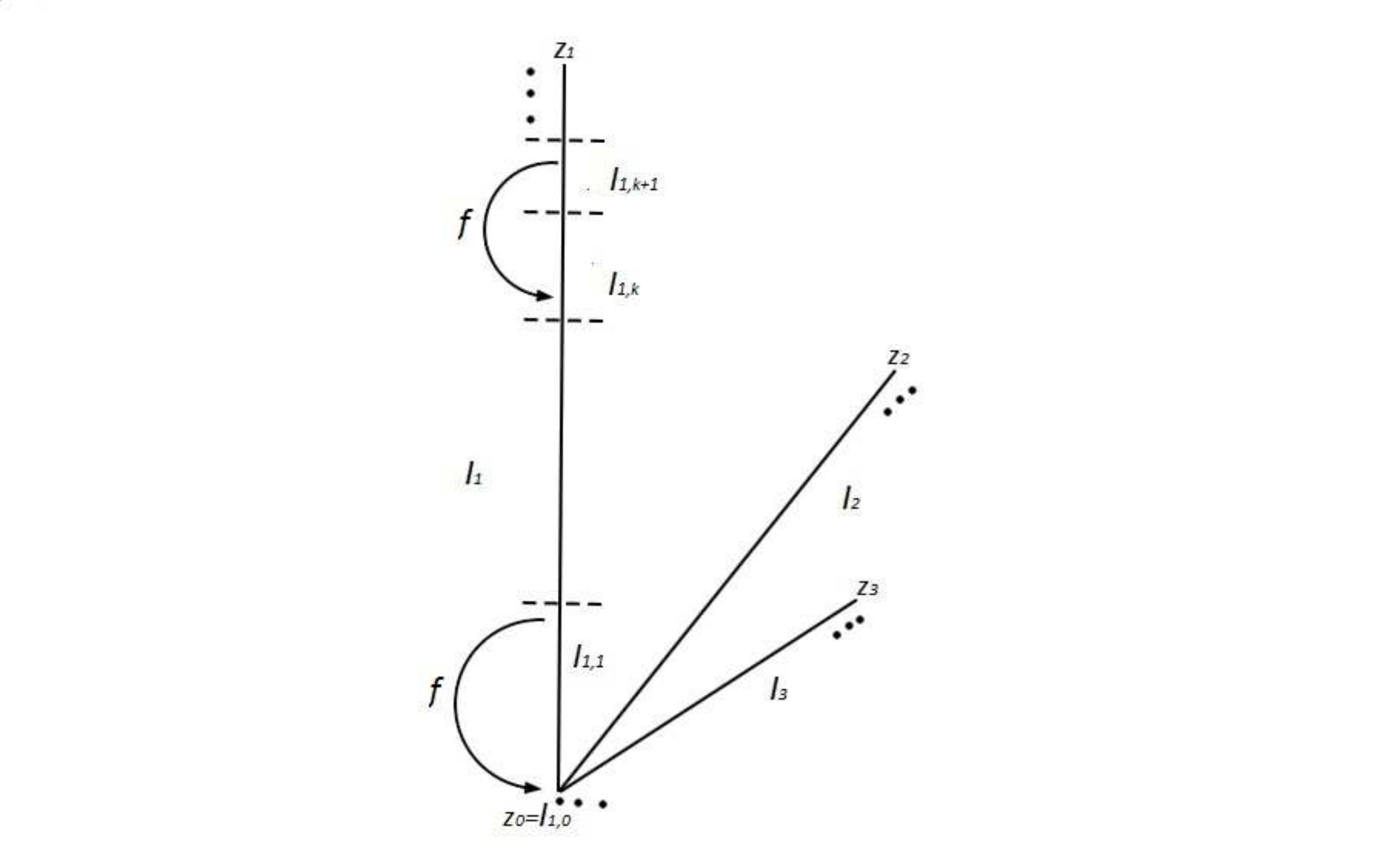}
		\caption{The map $f$ on the infinite star $X$}
	\end{figure}
	\medskip
	
	\begin{proof}
		The space $X$ is a countable union $X=\displaystyle\bigcup_{n\geq 0} I_{n}$ of arcs $I_n$ in $\mathbb{R}^{2}$ for each of which $z_{0}$ is one of its endpoints, $z_{n}$ is the other endpoint of $I_{n}$ distinct from $z_{0}$, with $n\geq 1$, such that for any $n\neq m$, $I_n\cap I_{m}=\{z_{0}\}$ and $\underset{n\to +\infty}\lim\textrm{diam}(I_{n})=0$. The space $X$, known as the $\omega$-star, is then a regular curve. The $I_{n},\; n\geq 1$ are the branches of $X$ and $z_{0}$ is its central point.  Denote by
		$I_{0} = \{z_{0}\}$. Let the map $g:[0,1]\to [0,1],\; x\longrightarrow \max\{0,2x-1\}$. We extend $g$ into a map of $X$ such that the restriction of $f$ on each branch is $g$, where the central point $\{z_{0}\}$ is playing the role of $0$.
		%
		%
		%
		
		It is plain that the map $f$ is continuous monotone and onto. Moreover, it satisfies the properties (i),(ii) and (iii) above. Indeed :
		
		\rm{(i)} We have $\displaystyle\bigcup_{n\geq 0} f^{-n}(z_{0})=X\setminus \{z_{n}: n\geq 1\}$, therefore $\omega_{f}(x)=\{z_{0}\}$. Notice that $x$ has only one negative orbit $(x_{k})_{k\geq 0}$ such that for any $k\geq 0,\; x_{k}\in I_{n_{x}}$. Hence $\alpha_{f}(x) =\{z_{n_{x}}\}$.\\
		
		\rm{(ii)} For any $n\geq 1$, we have $\{f(z_{n})\}=\{f^{-1}(z_{n})\}=\{z_{n}\}$. So (ii) follows.\\
		\rm{(iii)} $\omega_{f}(z_{0})=\{z_{0}\}$. Indeed, $z_{0}$ is a fixed point and $\alpha_{f}(z_{0})=\overline{\displaystyle\bigcup_{n\geq 0} f^{-n}(z_{0})}=X$. Notice that any negative orbit $(x_{k})_{k\geq 0}$ of $z_{0}$ is either the constant sequence $(x_{k}=z_{0})_{k\geq 0}$ or it coincides with the negative orbit of some $x\in X\setminus \{z_{n}: n\geq 0\}$, for $k$ large enough.
	\end{proof}			

In \cite{am}, Mchaalia proved for a monotone regular curve self-mapping $f$ of $X$, that every totally periodic $\omega$-limit set is finite, that is a periodic orbit. The following corollary shows analogously that every totally periodic $\alpha$-limit set is a periodic orbit.

\begin{cor}\label{pa}
Let $X$ be a regular curve, $f$ a monotone self mapping of $X$ and $x\in X_{\infty}$. If $\alpha_{f}(x)\subset P(f)$, then $\alpha_{f}(x)$ is a periodic orbit.
\end{cor}

\begin{proof}
Assume that $\alpha_{f}(x)\subset P(f)$. If $x\in X_{\infty}\setminus P(f)$, then by Theorem \ref{t31}, $\alpha_{f}(x)$ is minimal, and hence it is a periodic orbit. If $x\in P(f)$, then
$\alpha_{f}(x) = \overline{\displaystyle\bigcup_{n\geq 0}f^{-n}(x)}$. Moreover, we have $\displaystyle\bigcup_{n\geq 0}f^{-n}(x)=O_{f}(x)$; indeed, if $y\in \displaystyle\bigcup_{n\geq 0}f^{-n}(x)$, there is $n\geq 0$ such that $f^{n}(y)=x$ and moreover $y\in P(f)$. As $f^{n}(y) = f^{n}(f^{sp-n})(x)$, where $p$ is the period of $x$ and $s\geq 1$ an integer such that $sp-n\geq 0$, then $y= f^{sp-n}(x)\in O_{f}(x)$ (since the restriction $f_{|P(f)}$ is injective). We conclude that $\alpha_{f}(x)=O_{f}(x)$ a periodic orbit.
\end{proof}
\medskip

\medskip

\textbf{Conflict of interest}\\

The authors declare that there is no conflict of interests regarding the publication of this article. 
\\

\textbf{Data availability statement} \\

All data analysed in this study are included in this article.
\\

\textbf{Acknowledgements} \\

This work was supported by the research unit: ``Dynamical systems and their applications'', (UR17ES21), Ministry of
Higher Education and Scientific Research, Faculty of Science of Bizerte, Bizerte, Tunisia.
\bigskip

\bibliographystyle{amsplain}

\bigskip

\begin{thebibliography}{9}
\bibitem{a} H. Abdelli, \emph{ $\omega$-limit sets for  monotone local dendrite maps},  Chaos, Solitons Fractals, \textbf{71} (2015), 66--72.
\bibitem{HM} H. Abdelli, H. Marzougui, \emph{Invariant sets for monotone local dendrites},
Internat. J. Bifur. Chaos Appl. Sci. Engrg., \textbf{26} (2016), 1650150.
\bibitem{aam} H. Abdelli, H. Abouda and  H. Marzougui, \emph{Nonwandering points of monotone local dendrite maps revisited}, Topology Appl.,\textbf{ 250} (2018), 61--73.
\bibitem{an} G. Askri and I. Naghmouchi, \emph{On totally periodic $\omega$-limit sets in regular continua}, Chaos Solitons Fractals \textbf{75} (2015), 91--95.
\bibitem{bgl} F. Balibrea, J.L. Guirao and M. Lampart, \emph{A note on the definition of $\alpha$-limit set}, Appl. Math. Inf. Sci. \textbf{7} (2013), 1929--1932.
\bibitem{dev} A.F. Beardon, Iteration of rational functions: Complex analytic dynamical systems,
Graduate Texts in Mathematics, vol. 132, Springer-Verlag, New York, 1991.
\bibitem{Bldevan} P. Blanchard, R.L. Devaney, D. Look, M. Moreno Rocha, P. Seal, S. Siegmund, and D. Uminsky,
\emph{Sierpinski carpets and gaskets as Julia sets of rational maps}, In Dynamics on the Riemann
Sphere, European Mathematical Society, Z\"{u}rich, 2006, pp. 97--119.
\bibitem{Blokh}
A. Blokh, A. Bruckner, P. Humke, J. Sm\'{\i}tal, \textit{The space of $\omega$-limit sets of continuous map of the interval}, Trans. Amer. Math. Soc., \textbf{348} (1996), 1357--1372.
\bibitem{Aym} A. Daghar, \emph{ On regular curve homeomorphisms without periodic points}, J. Difference Equ. Appl. (2021), DOI: 10.1080/10236198.2021.1912030 
\bibitem{ka} H. Kato, \emph{Topological entropy of monotone maps and confluent maps on regular curves}, Topol.
Proc. \textbf{28} (2004), 587--593.
\bibitem{ka2} H. Kato, \emph{Topological entropy of piecewise embedding maps on regular curues}, Ergodic Theory Dyn. Syst. \textbf{26} (2006), 1115--1125.
\bibitem{Kur} K. Kuratowski, Topology, vol.2, Academic Press, New-York, 1968.
\bibitem{LFS} A. Lelek, \textit{On the topology of curves. II}, Fund. Math. \textbf{70} (1971), 131--138.
\bibitem{mak} E.N. Makhrova, \textit{On Limit Sets of Monotone Maps on Dendroids}, Applied Mathematics and Nonlinear Sciences \textbf{5} (2020) 311--316.
\bibitem{MI} H. Marzougui and I. Naghmouchi, \emph{On totally periodic $\omega$-limit sets}, Houston J. Math. \textbf{43} (2017) 1291--1303.
\bibitem{am} A. Mchaalia, On totally periodic $\omega$-limit sets for monotone maps on regular curves.  (2021). hal-03018900v2
\bibitem{Nadler} S.B. Nadler, Continuum Theory: An Introduction, (Monographs and Textbooks in Pure and Applied Mathematics, 158). Marcel Dekker, Inc., New York, 1992.
\bibitem{n} I. Naghmouchi, \textit{Homeomorphisms of regular curves}, J. Difference Equ. Appl., \textbf{23} (2017), 1485--1490.
\bibitem{n2} I. Naghmouchi, \textit{Dynamics of Homeomorphisms of regular curves}, Colloquium Math., \textbf{162} (2020), 263--277.
\bibitem{Nag3} I. Naghmouchi. \textit{Dynamics of monotone graph, dendrite and dendroid maps}. Internat. J. Bifur. Chaos Appl. Sci. Engrg., \textbf{21} (2011), 3205-3215.
\bibitem{Seidler} G.T. Seidler, \emph{The topological entropy of homeomorphisms on one-dimensional continua}, Proc. Amer. Math. Soc.
\textbf{108} (1990), 1025--1030.
\end{thebibliography}

\end{document}